\documentclass[12pt,reqno]{amsart}  
\usepackage{latexsym}
\usepackage{soul}
\usepackage{amssymb}
\usepackage{mathrsfs}
\usepackage{amsmath}
\usepackage{fancybox,color}
\usepackage{enumerate} 
\usepackage[latin1]{inputenc}
\usepackage{soul}

\usepackage{amscd,amsthm,amsxtra, array,ulem}

\usepackage[dvipsnames]{xcolor}

\usepackage[
colorlinks=true, 
linkcolor=magenta, 
citecolor=magenta,
pagebackref,
]{hyperref}

\usepackage{comment}

\numberwithin{equation}{section}

\usepackage{color}

\usepackage[left=2.2cm,top=2.32cm,right=2.2cm]{geometry}

\def\1{\raisebox{2pt}{\rm{$\chi$}}}

\newtheorem{theorem}{Theorem}[section]

\newtheorem{lemma}[theorem]{Lemma}

\newtheorem{remark}[theorem]{Remark}

\newcommand{\R}{{\mathbb R}}

\newcommand{\N}{{\mathbb N}}

\newcommand{\strt}[1]{\rule{0pt}{#1}} %
\newcommand{\car}[1]{\chi_{\strt{1.5ex}#1}}  %
\newcommand{\pa}[1]{\left(#1\right) } %
\newcommand{\set}[2]{\left\{#1 :#2\right\}} %

\def\1{\raisebox{2pt}{\rm{$\chi$}}}

\newcommand{\abs}[1]{\left|#1\right|}

\newcommand{\dif}{\mathrm{d}}

\def\vint_#1{\mathchoice%
        {\mathop{\kern 0.2em\vrule width 0.6em height 0.69678ex depth -0.58065ex
                \kern -0.8em \intop}\nolimits_{\kern -0.4em#1}}%
        {\mathop{\kern 0.1em\vrule width 0.5em height 0.69678ex depth -0.60387ex
                \kern -0.6em \intop}\nolimits_{#1}}%
        {\mathop{\kern 0.1em\vrule width 0.5em height 0.69678ex
            depth -0.60387ex
                \kern -0.6em \intop}\nolimits_{#1}}%
        {\mathop{\kern 0.1em\vrule width 0.5em height 0.69678ex depth -0.60387ex
                \kern -0.6em \intop}\nolimits_{#1}}}
\def\vintslides_#1{\mathchoice%
        {\mathop{\kern 0.1em\vrule width 0.5em height 0.697ex depth -0.581ex
                \kern -0.6em \intop}\nolimits_{\kern -0.4em#1}}%
        {\mathop{\kern 0.1em\vrule width 0.3em height 0.697ex depth -0.604ex
                \kern -0.4em \intop}\nolimits_{#1}}%
        {\mathop{\kern 0.1em\vrule width 0.3em height 0.697ex depth -0.604ex
                \kern -0.4em \intop}\nolimits_{#1}}%
        {\mathop{\kern 0.1em\vrule width 0.3em height 0.697ex depth -0.604ex
                \kern -0.4em \intop}\nolimits_{#1}}}

\newcommand{\aveint}[2]{\mathchoice%
        {\mathop{\kern 0.2em\vrule width 0.6em height 0.69678ex depth -0.58065ex
                \kern -0.8em \intop}\nolimits_{\kern -0.45em#1}^{#2}}%
        {\mathop{\kern 0.1em\vrule width 0.5em height 0.69678ex depth -0.60387ex
                \kern -0.6em \intop}\nolimits_{#1}^{#2}}%
        {\mathop{\kern 0.1em\vrule width 0.5em height 0.69678ex depth -0.60387ex
                \kern -0.6em \intop}\nolimits_{#1}^{#2}}%
        {\mathop{\kern 0.1em\vrule width 0.5em height 0.69678ex depth -0.60387ex
                \kern -0.6em \intop}\nolimits_{#1}^{#2}}}

\newcommand{\loc}{\mathrm{loc}}

\newcommand{\vertii}[1]{{\left\vert\kern-0.25ex\left\vert\kern-0.25ex  #1 
    \kern-0.25ex\right\vert\kern-0.25ex\right\vert}}

\title{On the weighted  inequality between the Gagliardo and Sobolev seminorms}

\author[R. Hurri-Syrj\"anen]{Ritva Hurri-Syrj\"anen}
\address[Ritva Hurri-Syrj\"anen]{Department of Mathematics and Statistics, Pietari Kalmin katu 5, FI-00014 University of Helsinki, Finland}
 \email{ritva.hurri-syrjanen@helsinki.fi}

\author[J. C. Mart\'inez-Perales]{Javier C. Mart\'inez-Perales}
\address[Javier C. Mart\'inez-Perales]{Calle Nueva, 18, Manilva, M\'alaga, Spain} \email{javicemarpe@gmail.com}

\author[C. P\'erez]{Carlos P\'erez}
\address[Carlos P\'erez]{Department of Mathematics, University of the Basque Country, IKERBASQUE 
(Basque Foundation for Science) and
BCAM \textendash  Basque Center for Applied Mathematics, Bilbao, Spain}
\email{cperez@bcamath.org}

\author[A. V.\! V\"ah\"akangas]{Antti V. V\"ah\"akangas}
\address[Antti V. V\"ah\"akangas]{University of Jyvaskyla, Department of Mathematics and Statistics, P.O. Box 35, FI-40014 University of Jyvaskyla, Finland} 
\email{antti.vahakangas@iki.fi}

\makeatletter
\@namedef{subjclassname@2020}{{\mdseries 2020} Mathematics Subject Classification}
\makeatother


\newcommand{\vertiii}[1]{{\left\vert\kern-0.25ex\left\vert\kern-0.25ex\left\vert #1 
    \right\vert\kern-0.25ex\right\vert\kern-0.25ex\right\vert}}

\begin{document}

\thanks{
C. P. is supported by grant  PID2020-113156GB-I00, Spanish Government; by the Basque Government through grant IT1247-19 and the BERC 2014-2017 program  and by the BCAM Severo Ochoa accreditation CEX2021-001142-S, Spanish Government. He is also very grateful to Department of Mathematics of the University of Jyvaskyla where the 3rd author was a visiting faculty and where this research was carried out. }

\begin{abstract} 
We prove  weighted inequalities between the Gagliardo and Sobolev seminorms
and also between the Marcinkiewicz quasi-norm and the Sobolev seminorm.
With  $A_1$ weights we improve earlier results of Bourgain, Brezis, and Mironescu.
\end{abstract}

\maketitle

%
%

%

\section{Introduction}

The classical $(1,p)$-Poincar\'e inequality for $1\le p<\infty$ establishes the existence of a dimensional constant $C(n)>0$  such that, for any continuously differentiable function $u$, the inequality
\begin{equation}\label{1-pPI}
\vint_Q \lvert u(x)-u_Q\rvert\, \dif x  \le  C(n)\, \ell(Q)\left(\vint_Q |\nabla u(x)|^p\,\dif x\right)^{1/p},
\end{equation}
holds for any cube $Q$, that is, for any cartesian product of $n$ intervals of the same side length $\ell(Q)$. Here, and along the rest of the paper,  $f_Q$ and $\vint_Q f$ denote the average of the function $f$ over $Q$, that is $f_Q=\vint_Q f=\lvert Q\rvert^{-1}\int_Q f(x) \dif x$. 

The importance of inequalities like  \eqref{1-pPI} in both weighted and unweighted setting does not need any justification, since they are among the most important tools in the theory of partial differential equations and Sobolev spaces. 
We refer to \cite{HMPV} for  a more complete list of references following our point of view.

{More} recently, fractional versions of these Poincar\'e inequalities  have attracted the attention of many authors. Indeed, on one hand, one can easily prove the existence of a constant $C(n)>0$ such that, for any locally integrable function $u$, the fractional $(1,p)$-Poincar\'e inequality%
\begin{equation}\label{FPI-rough}
\vint_Q \lvert u(x)-u_Q\rvert\,\dif x \leq  C(n) \ell(Q)^\delta\left(\vint_Q\int_Q\frac{|u(x)-u(y)|^p}{|x-y|^{n+\delta p}}\,\dif y\,\dif x\right)^{1/p}
\end{equation}
holds for any cube $Q$ in $\mathbb{R}^n$, for any $p\geq 1$ and for any $\delta>0$. Nevertheless, it turns out that \eqref{FPI-rough}  is far from being optimal. 
Indeed, the results in  \cite{BBM2} and a simple scaling argument  shows the existence of a constant $C(n)$ such that the highly interesting estimate
%
\begin{equation} \label{FPI with gain}
\vint_Q |u(x)-u_Q|\,\dif x\leq  C(n)\,\frac{(1-\delta)^{1/p}}{(n-\delta p)^{1/p'}} \ell(Q)^\delta\left(\vint_Q\int_Q\frac{|u(x)-u(y)|^p}{|x-y|^{n+\delta p}}\,\dif y\,\dif x\right)^{1/p}.
\end{equation}
holds for any cube $Q$ in $\mathbb{R}^n$, for any $1/2\leq \delta<1$ and for any  $1\leq p<n/\delta$. 
Observe that the factor $(1-\delta)^{1/p}$ can be very close to zero, providing an extra gain which may be thought as a new type of self-improving phenomenon. Actually, the factor $(1-\delta)^{1/p}$ is needed to avoid the loss of information that occurs when the parameter $\delta$ is close to $1$, as observed in \cite{B}.
In \cite{MS1} a different PDE approach on these inequalities was found, and in \cite{Mi}  
a very interesting general method combining interpolation and extrapolation of functions was given.
We  also refer to the recent paper \cite{DM2}.

On the other hand, according to \cite{BBM1},
 it turns out that the right hand side of \eqref{FPI with gain} satisfies
\begin{equation}\label{e.A}
\ell(Q)^\delta \left(\vint_Q\int_Q\frac{|u(x)-u(y)|^p}{|x-y|^{n+\delta p}}\,\dif y\,\dif x\right)^{1/p} \leq \frac{C(n)}{(1-\delta)^{1/p}} \ell(Q)\left(\vint_Q |\nabla u(x)|^p\,\dif x\right)^{1/p}
\end{equation}
for some dimensional constant $C(n)>0$.
%
Hence, the $(1,p)$-Poincar\'e inequality \eqref{1-pPI} 
for 
$1\le p<n/\delta$  with any $\delta \in [1/2, 1)$ 
can be proved as the result of the combination of the two intermediate inequalities \eqref{FPI with gain} and  \eqref{e.A}.

\section{Statement of main results}\label{sec:main}

This paper is an outgrowth of our work \cite{HMPV} where we proved weighted extensions of inequality \eqref{FPI with gain}  in the case of $A_1$ weights. 
We omit their formulation here. 
The purpose of this paper is to extend inequalities like \eqref{e.A} to the context of $A_1$ weights.

\begin{theorem}\label{weightedBBMp>1}  Let  $0<\delta< 1$,  $1<p< \infty$, $w\in A_1$ and $n\ge 2$.  There exists a dimensional constant $C(n)>0$  such that, for any cube $Q$ of $\mathbb{R}^n$ and any  $u\in C^1(Q)$, the inequality 
%
\begin{equation*}
\ell(Q)^{\delta} \left(\int_Q \int_{Q} \frac{\lvert u(x)-u(y)\rvert^p}{\lvert x-y\rvert^{n+\delta p}}\,\dif y\,w(x)\,\dif x\right)^{\frac{1}{p}}\leq  \frac{C(n)p'}{{ (1-\delta)^\frac{1}p}} [w]_{A_1}^{\frac{2}{p}}\ell(Q) 
\left(\int_{Q} \lvert \nabla u(x)\rvert^p\,w(x)\,\dif x\right)^{\frac{1}{p}}.
\end{equation*}
holds.
\end{theorem}
This is an interesting result since it provides a weighted inequality between the fractional and classical Sobolev norms. It is a weighted extension of the more classical inequality given in  Section~\ref{Appendix}. Since the weighted $(1,p)$-Poincar\'e inequality holds for any Muckenhoupt weight in $A_p$ (see \cite{FKS}), we believe that the  above result should also hold for any $A_p$ weight. 

\begin{theorem}\label{weightedBBMp=1} Let  $0<\delta< 1$, $w\in A_1$ and $n\ge 2$. There exists a dimensional constant $C(n)>0$  such that, for any cube $Q$ in $\R^n$  and any $u\in C^1(Q)$, the inequality
\begin{equation*}
\ell(Q)^{\delta} \int_Q \int_{Q} \frac{\lvert u(x)-u(y)\rvert}{\lvert x-y\rvert^{n+\delta}}\,\dif y\,w(x)\,\dif x   \leq C(n)\, \frac{[w]_{A_1}}{(1-\delta)^2} \ell(Q)
\, \int_Q |\nabla u(x)|w(x)\,\dif x
\end{equation*}
holds.

\end{theorem}

 We prove variants of Theorem \ref{weightedBBMp>1} and Theorem \ref{weightedBBMp=1} also for Borel measures $\mu$ in general.
We refer to inequality \eqref{e.general1} and
Section \ref{s.ps}.

We believe that neither of the two theorems above are  fully  sharp since the constant  $[w]^2_{A_1}$ we get is quadratic if $p>1$ while in the case $p=1$  the constant is linear. However,  the factor  $(1-\delta)^{-2}$ in the case $p=1$ is worse  and we believe that $(1-\delta)^{-1}$ should be the right constant in the inequality, as 
it is in the unweighted case. Actually we have the following weak type result in which the conjectured constant is obtained, although the method of proof gives a higher power of  $[w]_{A_1}$.

\begin{theorem}\label{weightedBBMp=1-weak}
Let  $0<\delta< 1$,  $w\in A_1$ and $n\ge 2$.  Then there exists a constant $C(n)>0$  such that, for any cube $Q$ in $\R^n$ and any $u\in C^1(Q)$, the inequality
\begin{equation*}
\ell(Q)^{\delta} \,\bigg{\|} \frac{u(x)-u(y)}{\lvert x-y\rvert^{n+\delta}} \bigg{\|}_{L^{1,\infty}\big(Q \times Q,  w(x)\,\dif x \times \dif y\big)}
\leq C(n)\,
 \frac{[w]_{A_1}^{ 2+\frac{1-\delta}{n}} }{ \delta(1-\delta)}   
 \ell(Q)  \int_Q |\nabla u(x)|w(x)\,\dif x.
\end{equation*}
holds.
\end{theorem}


For the unweighted case we refer to the famous paper \cite{BVSY}. We also refer to the related interesting results in \cite{DM1}.
 In any case, a result like the one \cite{B} is not known for us in a different setting than the Euclidean one. It may happen that the right constant in terms of $\delta$ is different from the one in the Euclidean case, namely $(1-\delta)^{-1}$.

\section{Muckenhoupt weights, maximal operators and Riesz potentials} \label{prelimin}

We review here some definitions and known results on the theory of Muckenhoupt weights, the Marcinkiewicz norm, maximal operators and fractional integrals. It will be useful here to denote by $\mathcal{Q}$ the set of all cubes in $\mathbb{R}^n$.

\subsection{Muckenhoupt weights}
We start by recalling some definitions and known results about Muckenhoupt weights. A weight is a function $w\in L^1_{\mathrm{loc}}(\mathbb{R}^n)$ satisfying  $w(x)> 0$  for almost every point  $x\in\mathbb{R}^n$. When $1<p<\infty$, we say that a weight $w\in L^1_{\mathrm{loc}}(\mathbb{R}^n)$  is in the $A_p$ class if its $A_p$ constant, 
\[
[w]_{A_p} = \sup_{Q\in \mathcal{Q}} \vint_Q w(x)\,\dif x\left(\vint_Q w(x)^{1-p'}\dif x\right)^{p-1},
\]
is finite. In case $p=1$, we say that $w\in A_1$ if there is a constant $C>0$ such that, for any cube $Q$ in $\mathbb{R}^n$,
\[
\frac{1}{|Q|} \int_Q w(x)\,\dif x\leq C\,\mathrm{ess\ inf}_{x\in Q}w(x),
\]
and the $A_1$ constant $[w]_{A_1}$ is defined as the smallest of these constants $C$.

Given $1\leq p<\infty$, it turns out (see \cite[p. 396]{GCRdF}) that a weight $w\in L^1_{\mathrm{loc}}(\mathbb{R}^n)$ is in the $A_p$ class if and only if there is a constant $C>0$ such that, for every cube $Q$ in $\mathbb{R}^n$ and every nonnegative measurable function $u$ on $Q$, the inequality
\begin{equation}\label{eq:prop-Ap-function}
\frac{1}{|Q|}\int_Q u(x)\ \dif x\leq C\left( \frac{1}{w(Q)}\int_Q u(x)^pw(x)\ \dif x\right)^{\frac{1}{p}}
\end{equation}
holds. Here, we denote the 
weighted measure of a measurable set $E$ by $w(E)=\int_Ew(x)\,\dif x$.  Moreover, the smallest constant $C$ satisfying the above inequality is precisely $[w]_{A_p}^\frac{1}{p}$.
Given an $A_p$ weight, a cube $Q\in \mathcal{Q}$ and a measurable subset $E\subset Q$, one can apply inequality \eqref{eq:prop-Ap-function} to the function $u= \chi_{E}$ to get
\begin{equation}\label{eq:prop-Ap-set}
\frac{|E|}{w(E)^{\frac1p}}\leq[w]^{\frac{1}{p}}_{A_p} \frac{|Q|}{w(Q)^{\frac1p}}.
\end{equation}

\subsection{Marcinkiewicz norms and Kolmogorov's inequality}

Similarly, we will use the standard notation for the normalized  Marcinkiewicz quasinorms: for any $0<q<\infty$, any measurable set $E\subset \R^n$ and a Borel measure $\mu$, we define
%
\begin{equation*}
\| u \|_{L^{q,\infty}\big(E, \frac{\dif \mu}{\mu(E)}\big)} =  \sup_{t>0}  t \, \left( \frac{1}{\mu(E)}\mu(\{x\in E:  |u(x)|>t\}) \right)^{\frac1q}.
\end{equation*}

We shall use the following Kolmogorov's inequality: given a Borel measure $\mu$, we have that, for every $0<q<r<\infty$ and every nonnegative measurable function $u$ on a cube $Q$,  
\begin{equation}\label{kolmogorov}
\frac{1}{\mu(Q)} \int_{Q} u(x)^{q}  \dif \mu(x)\leq 
\frac{r}{r-q}\, \| u \|^q_{L^{r,\infty}\big(Q, \frac{\dif \mu}{\mu(Q)}\big)} %
\end{equation}
See \cite[p. 485]{GCRdF}, for instance.

\subsection{Maximal functions and Riesz potentials}

Let $\mu$ be a Borel measure in $\R^n$. Let $\alpha\geq 0$. We will denote by   $M^c_\alpha\mu$ the fractional centered Hardy--Littlewood maximal function of $\mu$ on cubes, which is defined by 
\begin{equation}\label{eq:weighted-maximal-function}
M^c_\alpha\mu(x):=\sup_{\ell>0}\ell^\alpha\,\frac{\mu(Q(x,\ell))}{|Q(x,\ell)|},\qquad x\in\R^n,
\end{equation}
where $Q(x,\ell)$ is the cube of side length $\ell$ centered at $x$ and the supremum is taken over all $\ell>0$. The case $\alpha=0$ corresponds to the usual centered Hardy--Littlewood maximal function, which we simply denote by $M^c\mu$. We remove the superscript $c$ from the notation when the supremum is taken over all cubes $Q$ in $\mathbb{R}^n$ satisfying $x\in Q$. This allows to define the fractional and classical non-centered Hardy--Littlewood maximal functions $M_\alpha\mu$ and $M\mu$, respectively, by
\begin{equation}\label{eq:weighted-maximal-function-non-centered}
M_\alpha\mu(x):=\sup_{\mathcal{Q}\ni Q\ni x}\ell(Q)^\alpha\,\frac{\mu(Q)}{|Q|},\qquad x\in\R^n.
\end{equation}
When a cube $Q_0$ is given, we define  $M_{\alpha,Q_0}\mu$ by replacing the supremum in \eqref{eq:weighted-maximal-function-non-centered}  by the supremum over all cubes $Q$ in the class $\mathcal{Q}(Q_0)$ of all cubes  $Q\in\mathcal{Q}$ satisfying $Q\subset Q_0$. When $\mu$ is given by the integral of the absolute value of a function $u\in  L^1_{\mathrm{loc}}(\mathbb{R}^n)$, that is, $\dif \mu=\lvert u\rvert\,\dif x$, we replace $\mu$ by $u$ in the notation.  We also define, for a given weight $w\in L^1_{\mathrm{loc}}(\mathbb{R}^n)$, the centered weighted maximal function $M_w^c u$ of $u\in  L^1_{\mathrm{loc}}(\mathbb{R}^n)$ by
\begin{equation}\label{eq:weighted-maximal-function-non-centered-weighted}
M_w^c u(x):=\sup_{\ell >0}\frac{1}{w((Q(x,\ell)))}\,\int_{Q(x,\ell)} |u(y)|\,\dif w(y),\qquad x\in\R^n.
\end{equation}
We will use the classical weighted Fefferman--Stein inequality \cite{FS}
for a Borel measure $\mu$ and $u\in L^1_{\loc}(\R^n)$,
\begin{equation}\label{weightedweaktype}
\|M u\|_{\strt{2ex}L^{1,\infty}(d\mu)} \le C(n) \,
 \int_{\R^{n}} |u(x)|\,M\mu(x)\,\dif x.
\end{equation}
The usual proof for weights, namely when $d\mu=wdx$, goes through using a covering lemma of Vitali type for instance.  
From \eqref{weightedweaktype} we deduce, using Marcinkiewicz interpolation,  that
\begin{equation}\label{FSHLStrong}
\|Mu\|_{\strt{2ex}L^{p}(d\mu)}    \leq C(n)\,p'\, \|u\|_{\strt{2ex}L^{p}(M\mu)}
\end{equation}
holds for every $p\in (1,\infty)$.

Another family of operators which we shall be using and which are closely related to the fractional maximal functions  are the fractional integral operators or Riesz potentials which, for $0<\alpha<n$, are defined for a Borel measure $\mu$ in $\R^n$ by
\[
I_{\alpha} \mu(x)=\int_{\R^n} \frac{  \dif \mu(y) }{\lvert x-y\rvert^{n-\alpha}}\,,\qquad x\in\R^n.
\]
Whenever the measure $\mu$ is given by a nonnegative function $u\in L^1_{\mathrm{loc}}(\mathbb{R}^n)$, that is, $\dif \mu(x)=u(x)\,\dif x$, we get the usual fractional integral operator.

\section{Estimates for operators and representation formulas}\label{s.toolbox}

In this section we establish some basic results which will be used in the sequel. We first prove several estimates involving fractional maximal functions and Riesz potentials. Then we prove some representation formulas relating the oscillations of the functions under study with the aforementioned operators.

\subsection{Estimates for fractional maximal functions and Riesz potentials}\label{ss.toolbox1}

We start with the following  boundedness result for
the local fractional maximal operator. 

 \begin{lemma}\label{l.riesz_A1}
 There exists a constant $C(n)$ such that, for any $0<\alpha<n$, any $w\in A_1$ and any $u$ nonnegative measurable function on a cube $Q_0\in\mathcal{Q}$,
\[
 \int_{Q_0} M_{\alpha,Q_0}u(x) \,w(x)\,\dif x\leq   
\,\frac{C(n)}{\alpha}\, [w]_{A_1}^{ 1+\frac{\alpha}{n}} \ell(Q_0)^\alpha   \int_{Q_0} u(x) \,w(x)\,\dif x.
\]
\end{lemma}

\begin{proof}
	Let $x\in Q_0$. We first show that for every $x\in Q_0$,
\begin{equation}\label{e.apu}
M_{\alpha,Q_0}u(x) \leq c_n\,[w]_{A_1}^{ 1+\frac{\alpha}{n} }\, \ell(Q_0)^{\alpha}  \left(  \frac{1}{w(Q_0)} \int_{Q_0} u(y) \,w(y)\dif y\right)^{\frac{\alpha}{n}}\,
 \left(M^c_{w} (\chi_{Q_0}u)(x) \right)^{1-\frac{\alpha}{n}}. 
\end{equation}
Indeed, let $\theta=\frac{\alpha}{n}\in (0,1)$.
Let $Q\subset Q_0$ be a cube with $x\in Q$. 
By using inequality \eqref{eq:prop-Ap-function}, we obtain
\[
\begin{split}
&\frac{\ell(Q)^\alpha}{|Q|}  \int_{Q} u(y) \,\dif y= 
|Q|^{\frac{\alpha}{n}}\,  \left( \frac{1}{|Q|} \int_{Q} u(y) \,\dif y\right)^{\theta} 
\left( \frac{1}{|Q|} \int_{Q} u(y) \,\dif y\right)^{1-\theta}\\
&\qquad \leq \,|Q|^{\frac{\alpha}{n}}\,
\left(\frac{ [w]_{A_1}}{w(Q)}  \int_{Q} u(y) \,w(y)\dif y\right)^{\theta} 
\left( \frac{C(n)}{|Q(x,2\ell(Q))|} \int_{Q(x,2\ell(Q))} \chi_{Q_0}(y)u(y) \,\dif y\right)^{1-\theta}\\
&\qquad \leq [w]_{A_1}^{\theta}\,|Q|^{\frac{\alpha}{n}}\,
\left(\frac{1}{w(Q)}  \int_{Q} u(y) \,\dif w(y) \right)^{\theta} 
\left( \frac{C(n)[w]_{A_1}}{w(Q(x,2\ell(Q)))} \int_{Q(x,2\ell(Q))} \chi_{Q_0}(y)u(y) \,\dif w(y)\right)^{1-\theta}\\
&\qquad \leq C(n)
[w]_{A_1} \left(  \frac{|Q|}{w(Q)}\right)^{\frac{\alpha}{n}}\,\,  \left(\int_{Q_0} u(y) \,\dif w(y) \right)^{\frac{\alpha}{n}}\, 
 \left(M^c_{w} (\chi_{Q_0}u)(x) \right)^{1-\frac{\alpha}{n}}.
\end{split}
\]
Now we can apply \eqref{eq:prop-Ap-set} to get
\[
\begin{split}	
	\frac{\ell(Q)^\alpha}{|Q|}  \int_{Q} u(y) \,\dif y &\leq C(n)
	[w]_{A_1} \left(  \frac{|Q|}{w(Q)}\right)^{\frac{\alpha}{n}}\,\,  \left(\int_{Q_0} u(y) \,\dif w(y) \right)^{\frac{\alpha}{n}}\, 
	 \left(M^c_{w} (\chi_{Q_0}u)(x) \right)^{1-\frac{\alpha}{n}}\\
	 &  \le 
	C(n)
	[w]_{A_1}  \left( [w]_{A_1} \frac{|Q_0|}{w(Q_0)}\right)^{\frac{\alpha}{n}}  \,\,  	\left(\int_{Q_0} u(y) \,\dif w(y) \right)^{\frac{\alpha}{n}}\, 
	\left(M^c_{w} (\chi_{Q_0}u)(x) \right)^{1-\frac{\alpha}{n}}\\
	&\le 
	 C(n)
	[w]_{A_1}^{1+\frac{\alpha}{n}}  \ell(Q_0)^\alpha \,  \left(\frac{1}{w(Q_0)}\int_	{Q_0} u(y) \,\dif w(y) \right)^{\frac{\alpha}{n}}\, 
	\left(M^c_{w} (\chi_{Q_0}u)(x) \right)^{1-\frac{\alpha}{n}}. 
\end{split}
\]

By taking supremum over
all cubes $Q\in \mathcal{Q}(Q_0)$ with $x\in Q$, we get inequality \eqref{e.apu}. 

Now we can apply inequality \eqref{e.apu} to obtain
\begin{align*}
&\frac{1}{w(Q_0)} \int_{Q_0} M_{\alpha,Q_0}u(x) \,\dif w(x) \\
&\leq c_n[w]_{A_1}^{1+\frac{\alpha}{n}}\, \ell(Q_0)^{\alpha}  
\left(\frac{1}{w(Q_0)} \int_{Q_0} u(x) \,\dif w(x) \right)^{\frac{\alpha}{n}}\,
\frac{1}{w(Q_0)} \int_{Q_0} \, \left(M^c_{w} (\chi_{Q_0}u)(x) \right)^{1-\frac{\alpha}{n}}\,\dif w(x).
\end{align*}

Since $M^c_w$ is of weak type $(1,1)$ with respect to $\dif w(x)$ and with norm depending just on $n$ by the Besicovitch covering lemma, we can apply Kolmogorov's inequality  \eqref{kolmogorov} to bound the last integral from above as follows
\[
\begin{split}
\frac{1}{w(Q_0)} \int_{Q_0} \, \left(M^c_{w} (\chi_{Q_0}u)(x) \right)^{1-\frac{\alpha}{n}}\,w(x)\,\dif x&\leq \frac{n}{\alpha}\, \| M^c_{w} (\chi_{Q_0}u)\|^{1-\frac{\alpha}{n}}_{L^{1,\infty}\big(Q_0, \frac{\dif w}{w(Q_0)}\big)}\\
& \leq \frac{C(n)}{\alpha}\, \left(\frac{1}{w(Q_0)} \int_{Q_0} u(x) \,w(x)\,\dif x\right)^{1-\frac{\alpha}{n}},
\end{split}
\]
By combining the above estimates, we get
$$
\frac{1}{w(Q_0)} \int_{Q_0} M_{\alpha,Q_0}u(x) \,w(x)\,\dif x\leq \frac{C(n)}{\alpha}\, [w]_{A_1}^{1+\frac{\alpha}{n}}\,
\frac{ \ell(Q_0)^{\alpha}  }{w(Q_0)} \int_{Q_0} u(x) \,w(x)\,\dif x
$$
which is the result we wanted to prove.
\end{proof}

The following lemma will be very useful. 
Inequality \eqref{lem-eq1} is probably known but we provide a proof for convenience of the reader. 
This  can also be found in \cite{GLP}.
Inequality \eqref{lem-eq2} is well-known and follows from  \cite{He}. 

\begin{lemma}\label{New-pointwise-B}
	Let $Q_0$ be a cube in $\R^n$, $\mu$ be a Borel measure, and $0<\alpha<n$. Then there
	is a constant $C(n)>0$ such that   the inequality
	\begin{equation} \label{lem-eq1}
	I_\alpha(\chi_{Q_0}\mu)(x) \leq \frac{C(n)}{\alpha} \mu(Q_0)^\frac{\alpha}{n}M^c( \chi_{Q_0}\mu)(x)^\frac{n-\alpha}{n} 
	\end{equation}
	holds  for every  $x\in \R^n$.  	
	 Furthermore, the inequality 
	\begin{equation} \label{lem-eq2}
	 I_\alpha(\chi_{Q_0}\mu)(x) \leq \frac{C(n)}{\alpha}  \ell(Q_0)^\alpha M(\chi_{Q_0}\mu)(x), 
	\end{equation}
	holds
	for every $x\in Q_0$,.
	%
	If $w\in A_1$, then the inequality
	\begin{equation} \label{lem-eq3}
	 I_\alpha(\chi_{Q_0}w)(x) 
	\leq  \frac{C(n)}{\alpha}[w]_{A_1}^\frac{n-\alpha}{n}w(Q_0)^\frac{\alpha}{n}w(x)^\frac{n-\alpha}{n}	
	\end{equation}
	holds  for almost every   $x\in \R^n$.

\end{lemma}

\begin{proof} 
	For $x\in\R^n$ and  $t>0$  we denote
	$$Q_{x,t}=Q\pa{x, 2 t^{-\frac{1}{n-\alpha}}}.$$
	By the layer-cake formula, we obtain
	\begin{align*}
	\int_{Q_0}\frac{\dif \mu(y)}{\abs{x-y}^{n-\alpha}}&= \int_0^{\infty} \mu\pa{\set{y\in Q_0}{ \frac{1}{\abs{x-y}^{n-\alpha}}>t}}\dif t\\
	&= \int_0^{\infty} \mu\pa{\set{y\in Q_0}{\abs{x-y}<t^{-\frac{1}{n-\alpha}}}}\dif t\\
	&\leq \int_0^{\infty} \min\left\{\mu(Q_0), \frac{\mu( Q_0\cap Q_{x,t})}{|Q_{x,t}|}\,|Q_{x,t}|\right\} \dif t\\
	&\leq  \int_0^{\infty}\min\left\{ \mu(Q_0),M^c(\chi_{Q_0}\mu)(x)\pa{ 2 t^{-\frac{1}{n-\alpha}}}^n\right\}\dif t\\
	&\le  \int_0^{\pa{\frac{M^c(\chi_{Q_0}\mu)(x)}{\mu(Q_0)}}^\frac{n-\alpha}{n}}\mu(Q_0)\,\dif t
	+C(n)\int_{\pa{\frac{M^c(\chi_{Q_0}\mu)(x)}{\mu(Q_0)}}^\frac{n-\alpha}{n}}^\infty M^c(\chi_{Q_0}\mu)(x)t^{-\frac{n}{n-\alpha}}\,\dif t\\
	&\le  \frac{C(n)}{\alpha}\mu(Q_0)^\frac{\alpha}{n}M^c(\chi_{Q_0}\mu)(x)^\frac{n-\alpha}{n},
	\end{align*}
	which is inequality \eqref{lem-eq1}. Observe that inequality \eqref{lem-eq3} is a direct consequence of \eqref{lem-eq1}. 
	
	To show \eqref{lem-eq2}, we use inequality \eqref{lem-eq1} with the measure $\chi_{Q_0}\mu$ to get
	\[
	  \begin{split}
		I_\alpha(\chi_{\strt{1.7ex} Q_0} \mu)(x) & \leq \frac{C(n)}{\alpha} \mu(Q_0)^\frac{\alpha}{n}M^c( \chi_{\strt{1.7ex} Q_0} \mu)(x)^\frac{n-\alpha}{n}\\
		& = \frac{C(n)}{\alpha} |Q_0|^\frac{\alpha}{n}\left(\frac{\mu(Q_0)}{|Q_0|}\right)^\frac{\alpha}{n}M^c(\chi_{ \strt{1.7ex} Q_0 } \mu)(x)^\frac{n-\alpha}{n}\\
		& \leq \frac{C(n)}{\alpha} |Q_0|^\frac{\alpha}{n} M ( \chi_{\strt{1.7ex} Q_0} \mu)(x) ^\frac{\alpha}{n}M^c(\chi_{\strt{1.7ex} Q_0}\mu)(x)^\frac{n-\alpha}{n}\\
		& = \frac{C(n)}{\alpha} \ell(Q_0)^\alpha M (\chi_{\strt{1.7ex} Q_0}\mu)(x)
	  \end{split}  
	\]
	for every $x\in Q_0$, since $M^c(\chi_{Q_0}\mu)(x)\le M(\chi_{Q_0}\mu)(x)$. This is the desired inequality \eqref{lem-eq2}.
	\end{proof}

\subsection{Representation formulas from Poincar\'e-type inequalities}

We present some local representation formulas which follow from general Poincar\'e-type inequalities. Then we will 
 introduce some consequences of these representation formulas.
The following lemma is a key  in our arguments. It essentially follows from 
  \cite{FLW} or  \cite{Ha}, but it can also be obtained by following the proof of \cite[Lemma 4.10]{HV}, since cubes are examples of John domains.  We are interested in the tracking of the constants involved in our estimates and so   we will provide the proof here for the sake of clarity.  

\begin{lemma}\label{l.loc_impro} 
Assume that $0<\beta\le \alpha<n$. Let $Q_0$ be a cube in $\R^n$. Suppose there exists a constant $\kappa>0$, a function $u\in L^1(Q_0)$ and a nonnegative measurable function $g$ such that
\begin{align*}
&\vint_{Q}\lvert u(x)-u_{Q}\rvert\,\dif x
 \leq  {\kappa}\, \ell(Q)^{\alpha}\vint_Q g(x)\,\dif x
\end{align*}
for every cube $Q\subset Q_0$.
Then there exists a dimensional constant $C(n)>0$  such that
\[
\lvert u(x)-u(y)\rvert
\leq C(n)\, \frac{ \kappa}{\beta}  \,\lvert x-y\rvert^{\beta}
\bigl(M_{\alpha-\beta,Q_0}(g\chi_{Q_0})(x)+M_{\alpha-\beta,Q_0}(g\chi_{Q_0})(y)\bigr)
\]
for every pair $x,y\in Q_0$ of Lebesgue points  of $u$.
\end{lemma}

\begin{proof}
 
Pick two Lebesgue points $x,y\in Q_0$ of $u$ and
let $R_0\subset Q_0$ be a closed cube such that $\ell(R_0)\le \lvert x-y\rvert$ and $x,y\in R_0$.
 For every $j\in\mathbb{N}$ there
exists a cube $R_j\subset R_{j-1}$ such that $x\in R_j$ and $\ell(R_j)=2^{-1}\ell(R_{j-1})$.
Since $x$ is a Lebesgue point of $u$, we can bound the oscillation of $u$ over $R_0$ by a telescopical sum of the distances between averages of $u$ over the cubes $R_j$, and then  use the Poincar\'e-type inequality as follows:
\begin{align*}
\lvert u(x)-u_{R_0}\rvert
&\le\sum_{i=0}^\infty\lvert u_{R_{i+1}}-u_{R_i}\rvert
\\
&\le\sum_{i=0}^\infty\frac{\lvert R_i\rvert}{\lvert R_{i+1}\rvert}
\vint_{R_i}\lvert u(z)-u_{R_i}\rvert\,\dif  z
\\
&\le  C(n)\,\kappa\,\sum_{i=0}^{\infty}(2^{-i}\ell(R_0))^{\beta}(2^{-i}\ell(R_0))^{\alpha-\beta}
\vint_{R_{i}} g(z)\,\dif  z
\\
&\le C(n)\frac{ {\,\kappa}\,}{1-2^{-\beta}}\ell(R_0)^\beta M_{\alpha-\beta,Q_0} (g\chi_{Q_0})(x).
\end{align*}

Repeating the argument with $y$ yields the estimate
\begin{align*}
\lvert u(x)-u(y)\rvert
&\le\lvert u(x)-u_{R_0}\rvert
+\lvert u(y)-u_{R_0}\rvert
\\
&\le C(n)\frac{ { \,\kappa}\,}{1-2^{-\beta}}\lvert x-y\rvert^\beta
\bigl(M_{\alpha-\beta,Q_0} (g\chi_{Q_0})(x)+M_{\alpha-\beta,Q_0} (g\chi_{Q_0})(y)\bigr),
\end{align*}
and this completes the proof, since $1-2^{-\beta}$ is comparable to $\beta$ as $\beta\to0$.
\end{proof}

\begin{lemma}\label{l.riesz_poinc}
Let $Q_0$ be a cube in $\R^n$.
Assume that  $0<\alpha< n$ and  consider $0<\eta <n-\alpha$ and $1\leq r<\infty $. Suppose that  there exists a constant $\kappa>0$, a function $u\in L^1(Q_0)$ and a nonnegative measurable function $g$ such that 
\begin{equation}\label{kappaAssump0}
\vint_{Q}\lvert u(x)-u_{Q}\rvert\,\dif x
 \leq  {\kappa} \,\ell(Q)^{\alpha}   \left(\vint_Q g(x)^r\,\dif x\right)^{\frac1r}
\end{equation}
for every cube $Q\subset Q_0$. Then there exists a dimensional constant $C(n)$
such that
\begin{equation}\label{eq:pointwise-I1}
\lvert u(x)-u_{Q_0}\rvert \leq 
C(n)\,  \frac{\kappa }{\alpha^{1/r'} \eta^{1/r}}
\,\ell(Q_0)^{\alpha/r'}
\left(I_\alpha(g^r\chi_{Q_0})(x)\right)^{^{\frac1r}} \\
\end{equation}
for every Lebesgue point $x\in Q_0$ of $u$.
\end{lemma}

\begin{proof} 
 
This result is well known but  we need to be precise with the main parameters involved. We adapt the main ideas from \cite{FLW} in the case $r=1$ and \cite{FH} when $r>1$ and we also refer to  \cite{LP}.   
Fix a Lebesgue point $x\in Q_0$ of $u$.  Then  there exists a chain $\{Q_k\}_{k\in\N}$ of nested  dyadic subcubes of $Q_0$ such that  $Q_1=Q_0$, $Q_{k+1}\subset Q_k$ with $|Q_k|=2^n |Q_{k+1}|$ for all   $k\in\N$ and
$\{x\}=\bigcap_{k\in\N} Q_k$.
Then, we can use an argument similar to the one in the proof of Lemma \ref{l.loc_impro} to obtain that
\begin{equation*}
\lvert u(x)-u_{Q_0}\rvert=\left \lvert\lim_{k\to\infty} u_{Q_k} - u_{Q_1}\right\rvert\le \sum_{k\in\N} \left\lvert u_{Q_{k+1}}- u_{Q_k}\right\rvert.
\end{equation*}
Using the dyadic structure of the chain and the Poincar\'e-type inequality 
\eqref{kappaAssump0}, we obtain that 
\begin{align*}
\sum_{k\in\N} \left\lvert u_{Q_{k+1}}- u_{Q_k}\right\rvert & \le  \sum_{k\in\N} \frac{1}{|Q_{k+1}|}\int_{Q_{k+1}}\lvert u(y)-u_{Q_k}\rvert\,\dif y\\
& \le  2^n \sum_{k\in\N} \frac{1}{|Q_{k}|}\int_{Q_{k}} \lvert u(y)-u_{Q_k}\rvert\,\dif y\\
& \le  2^n\kappa \sum_{k\in\N} \ell(Q_k)^\alpha\left(\frac{1}{|Q_{k}|}\int_{Q_{k}}g(y)^r\,\dif y \right)^{1/r}\\ 
 &\le  2^n\kappa \left(\sum_{k\in\N} \ell(Q_k)^\alpha\right)^{1/r'} \left(\sum_{k\in\mathbb{N}} \frac{\ell(Q_k)^\alpha}{|Q_{k}|}\int_{Q_{k}}g(y)^r\,\dif y \right)^{1/r}\\ 
& \leq \frac{2^n}{(1-2^{-\alpha})^{1/r'}}\kappa \ell(Q_0)^{\alpha/r'}\left(\int_{Q_0}g(y)^r \sum_{k\in\N} \ell(Q_k)^{\alpha-n}\chi_{Q_k}(y)\,\dif y\right)^{1/r}.
\end{align*}

Note that the immediate estimate $|x-y|\le \sqrt{n}\ell(Q_k)$ produces an extra unwanted $\log$ factor when summing the series. We instead proceed as follows. 
Fix $y\in Q_0\setminus \{x\}$ and pick $0<\eta< n-\alpha$. Write $k_0(y)=\max\{j\in \mathbb{N}: 2^{j-1}\le \sqrt{n}\frac{\ell(Q_0)}{|x-y|}\}$. Then 
\begin{align*}
\sum_{k\in\N} \ell(Q_k)^{\alpha-n}\chi_{Q_k}(y) &\le  \frac{C(n)}{|x-y|^{n-\alpha-\eta}}\sum_{k=1}^{k_0(y)} \ell(Q_k)^{-\eta}\chi_{Q_k}(y)\\
&\le \frac{C(n)}{|x-y|^{n-\alpha-\eta}\ell(Q_0)^\eta}\sum_{k=1}^{k_0(y)} 2^{(k-1)\eta}\\
& \le 
\frac{C(n)\,2^{\eta k_0(y)}}{|x-y|^{n-\alpha-\eta}\ell(Q_0)^\eta(1-2^{-\eta})}
\\&\le  \frac{C(n)}{|x-y|^{n-\alpha}(1-2^{-\eta})}. %
\end{align*}
We conclude that  the desired inequality
\begin{equation*}
\begin{aligned}
|u(x)-u_{Q_0}|&\le  \frac{2^n C(n)^{1/r}\kappa}{(1-2^{-\alpha})^{1/r'}(1-2^{-\eta})^{1/r}}  \ell(Q_0)^{\alpha/r'} \left(\int_{Q_0} \frac{g(y)^r}{|x-y|^{n-\alpha}}\,\dif y\right)^{\frac1r}\\
&\leq
\frac{C(n)\kappa}{\alpha^{1/r'}\eta^{1/r}}  \ell(Q_0)^{\alpha/r'} (I_\alpha(g^r\chi_{Q_0})(x))^{1/r}
\end{aligned}
\end{equation*}
holds for Lebesgue points $x\in Q_0$ of $u$.  
\end{proof}

For the case $\alpha=1$ we get a sharper result. 

\begin{lemma}\label{New-pointwise-A}
Let $Q_0$ be a cube in $\R^n$ with $n\ge 2$.  Suppose there exists a constant $\kappa>0$, a function $u\in L^1(Q_0)$ and a nonnegative measurable function $g$ such that 
\begin{align*}
&\vint_{Q}\lvert u(x)-u_{Q}\rvert\,\dif x
 \leq  {\kappa}\, \ell(Q)\vint_Q g(x)\,\dif x
\end{align*}
for every cube $Q\subset Q_0$.
Then there exists a dimensional constant $C(n)$  such that
\begin{equation} \label{new-LipscType}
\begin{aligned}
 \lvert u(x)-u(y)\rvert   \leq
C(n)\, \kappa  \,\lvert x-y\rvert \min_{z\in\{x,y\}} M(g\chi_{Q_0})(z)^{\frac1n}\max_{z\in\{x,y\}} M(g\chi_{Q_0})(z)^{\frac1{n'}}\\
\end{aligned}
\end{equation}
for every pair $x,y\in Q_0$ of Lebesgue points  of $u$.
\end{lemma}

\begin{proof}  
 
Pick two Lebesgue points $x,y\in Q_0$ of $u$,
and let $R_0\subset Q_0$ be a closed cube such that $\ell(R_0)\le \lvert x-y\rvert$ and $x,y\in R_0$. 
The representation formula  \eqref{eq:pointwise-I1} in 
Lemma \ref{l.riesz_poinc}, with $\alpha=r=1$, yields
\begin{equation*}%
\lvert u(x)-u_{R_0}\rvert \leq 
C(n)\,  \kappa \,
I_1(g\chi_{R_0})(x).
\end{equation*}
Next we use inequality \eqref{lem-eq1} in Lemma \ref{New-pointwise-B}  to derive inequality 
\begin{equation*}%
\lvert u(x)-u_{R_0}\rvert \leq 
C(n)\,  \kappa \, \left(\int_{R_0}g(z)\,\dif z\right)^\frac{1}{n}M^c(\chi_{R_0}g)(x)^\frac{1}{n'}.
\end{equation*}
Similar estimates hold for $y$.

Since $x,y\in R_0$, we get 
\begin{align*}
\lvert u(x)-u(y)\rvert  &\leq \lvert u(x)-u_{R_0}\rvert +\lvert u(y)-u_{R_0}\rvert \\
& \leq C(n)\, \kappa \, \left( \left(\int_{R_0}g(z)\,\dif z\right)^\frac{1}{n}M^c(\chi_{R_0}g)(x)^\frac{1}{n'} +  \left(\int_{R_0}g(z)\,\dif z\right)^\frac{1}{n}M^c(\chi_{R_0}g)(y)^\frac{1}{n'}
\right)\\
& =  C(n)\, \kappa \, \ell(R_0) \left(  \vint_{R_0}g(z) \chi_{Q_0}\,\dif z  \right)^\frac{1}{n} \left(
M^c(\chi_{Q_0}g)(x)^\frac{1}{n'} +  M^c(\chi_{Q_0}g)(y)^\frac{1}{n'} \right)\\
& \leq 2C(n)\, \kappa \lvert x-y\rvert 
\min_{z\in\{x,y\}} M(g\chi_{Q_0})(z)^{\frac1n}\max_{z\in\{x,y\}} M(g\chi_{Q_0})(z)^{\frac1{n'}}.
\end{align*}
This is the desired inequality \eqref{new-LipscType}.
\end{proof}

\section{Proofs of theorems \ref{weightedBBMp>1}, \ref{weightedBBMp=1}, and  \ref{weightedBBMp=1-weak}}

In this section we give the proofs of the main theorems of this paper. Each proof is given in a separate subsection.

\subsection{Proof of Theorem \ref{weightedBBMp>1}}
We first let $\mu$ be a Borel measure in $\R^n$; 
at the end of the proof, we will apply the obtained estimates in the case $\dif \mu = w\,\dif x$.
Fix a cube $Q$ in $\R^n$, $n\ge 2$,  and function
$u\in C^1(Q)$.
By applying the  $(1,1)$-Poincar\'e inequality \eqref{1-pPI}, there exists a dimensional constant
$C(n)>0$ such that
\[\vint_{R}\lvert u(x)-u_{R}\rvert\,\dif x \leq  C(n)\, \ell(R)\vint_R \lvert \nabla u(x)\rvert\,\dif x
\]
for every cube $R\subset Q$.
We use Lemma \ref{l.loc_impro} with parameters $\beta=\alpha=1$,
$g=|\nabla u|$ and $\kappa=C(n)$. 
Hence,  for any pair of Lebesgue points $x,y\in Q$ of $u$, we have
\[
|u(x)-u(y)|\leq  C(n)\,|x-y| (M_{Q}(g\chi_Q)(x)+M_{Q}(g\chi_Q)(y)).
\]
Since $u$ is continuous, we know that every point of $Q$ is a Lebesgue point of $u$. By applying the triangle inequality, Tonelli's theorem  and Lemma \ref{New-pointwise-B} twice, 
\[
\begin{split}
 \Bigg{(}\int_Q \int_{Q}  \frac{\lvert u(x)-u(y)\rvert^p}{\lvert x-y\rvert^{n+\delta p}}\,\dif y\, \,\dif \mu(x)\Bigg{)}^{\frac{1}{p}}
& \leq 
C(n)\left(\int_Q \int_{Q} \frac{M_{Q}(g\chi_Q)(x)^p}{\lvert x-y\rvert^{n-p(1-\delta)}}\,\dif y\,\,\dif \mu(x)\right)^{\frac{1}{p}}\\
&\quad+ 
C(n)\left(\int_Q \int_{Q} \frac{M_{Q}(g\chi_Q)(y)^p}{\lvert x-y\rvert^{n-p(1-\delta)}}\,\dif y\,\,\dif \mu(x)\right)^{\frac{1}{p}}\\
& \leq C(n)\frac{ \ell(Q)^{1-\delta} }{ p^{\frac{1}{p}}{(1-\delta)^\frac{1}p}}  \left(\int_Q M(g\chi_Q)(x)^p \,\,\dif \mu(x)\right)^{\frac{1}{p}}\\
&\quad+ 
C(n)\left(\int_Q M(g\chi_Q)(y)^p\int_{Q} \frac{ \dif \mu(x)}{\lvert x-y\rvert^{n-p(1-\delta)}} \,\dif y\,\right)^{\frac{1}{p}}\\
&  \leq C(n)\frac{ \ell(Q)^{1-\delta} }{ {p^{\frac{1}{p}}(1-\delta)^\frac{1}p}}  \left(\int_Q M(g\chi_Q)(x)^p \,\dif \mu(x)\right)^{\frac{1}{p}}\\
&\quad +  C(n)\frac{ \ell(Q)^{1-\delta} }{ {p^{\frac{1}{p}}(1-\delta)^\frac{1}p}} 
\left(\int_Q M(g\chi_Q)(y)^p\,M(\chi_Q\mu)(y) \,\dif y\,\right)^{\frac{1}{p}}.
\end{split}
\] 
Since $1<p<\infty$,  we can apply the Fefferman--Stein inequality \eqref{FSHLStrong} to get 
\[
\begin{split}
&\Bigg{(}\int_Q \int_{Q} \frac{\lvert u(x)-u(y)\rvert^p}{\lvert x-y\rvert^{n+\delta p}}\,\dif y\,\dif \mu(x)\Bigg{)}^{\frac{1}{p}}   \\
&\quad \leq 
C(n) p' \frac{\ell(Q)^{1-\delta} }{ (1-\delta)^\frac{1}p}  \left\{  \,\, \left(\int_Q g(x)^p \,M(\mu\chi_Q)(x) \dif x\right)^{\frac{1}{p}}
+
\left(\int_Q g(y)^p\,M^2(\mu\chi_Q)(y)\,\dif y\,\right)^{\frac{1}{p}}\right\}.
\end{split}
\]
Since $M(\mu\chi_Q) \leq M^2(\mu\chi_Q)$  almost everywhere, this finishes the proof of the inequality
\begin{equation}\label{e.general1}
\begin{split}
\ell(Q)^{\delta} \Bigg{(}\int_Q \int_{Q} & \frac{\lvert u(x)-u(y)\rvert^p}{\lvert x-y\rvert^{n+\delta p}}\,\dif y\,\dif \mu(x)\Bigg{)}^{\frac{1}{p}} \\&\quad   \leq \frac{C(n)\,p'}{ (1-\delta)^\frac{1}p} \ell(Q)  \left(\int_Q |\nabla u(x)|^p \,M^2(\chi_Q\mu)(x)\,\dif x\right)^{\frac{1}{p}}
\end{split}
\end{equation}
for Borel measures $\mu$ in $\R^n$. 
If $\dif \mu(x)= w(x)\,\dif x $ for some $w\in A_1$, then $M^2(w\chi_Q)\le [w]_{A_1}^2w$ almost everywhere, and this implies the second inequality in the theorem.

\subsection{Proof of Theorem \ref{weightedBBMp=1}}\label{s.ps}
We first let $\mu$ be a Borel measure in $\R^n$; 
at the end of the proof, we will apply the obtained estimate in the case $\dif \mu = w\,\dif x$.
Fix a cube $Q$ in $\R^n$ and a function
$u\in C^1(Q)$.
By the  $(1,1)$-Poincar\'e inequality \eqref{1-pPI}, there exists a dimensional constant
$C(n)>0$ such that
\[\vint_{R}\lvert u(x)-u_{R}\rvert\,\dif x \leq  C(n)\, \ell(R)\vint_R \lvert \nabla u(x)\rvert\,\dif x
\]
for every cube $R\subset Q$.
We use Lemma \ref{New-pointwise-A} with parameters 
$g=|\nabla u|$ and $\kappa=C(n)$. Hence, there exists a dimensional constant $C(n)$ such that, for every pair of Lebesgue points $x,y\in Q$ of $u$
\begin{align*}
\lvert u(x)-u(y)\rvert  \leq 
C(n)\,\lvert x-y\rvert  \min_{z\in\{x,y\}} M(g\chi_{Q_0})(z)^{\frac1n}\max_{z\in\{x,y\}} M(g\chi_{Q_0})(z)^{\frac1{n'}}.
\end{align*}
Let us define the set
\[
A=\{ (x,y)\in Q\times Q : M(g\chi_{Q})(x) \leq M(g\chi_{Q})(y) \}.
\]
Since every point of $Q$ is a Lebesgue point of $u\in C^1(Q)$, we get
\[
\begin{split}
 \int_{Q} \int_{Q} &\frac{\lvert u(x)-u(y)\rvert}{\lvert x-y\rvert^{n+\delta}}\,\dif y\,\dif \mu(x)\\
&\leq
\int_{Q} \int_{Q}   \frac{\min_{z\in\{x,y\}} M(g\chi_{Q_0})(z)^{\frac1n}\max_{z\in\{x,y\}} M(g\chi_{Q_0})(z)^{\frac1{n'}}}{\lvert x-y\rvert^{n-(1-\delta)}}\,\dif y\,\dif \mu(x)\\
&=\int_{Q}\int_Q \chi_{A}(x,y) \frac{M(g\chi_{Q})(x)^{\frac1n}\,M(g\chi_{Q})(y)^{\frac1{n'}}}{\lvert x-y\rvert^{n-(1-\delta)}}\,\dif y\,\dif \mu(x)\\
&\qquad\qquad +
\int_Q\int_Q \chi_{(Q\times Q)\setminus A}(x,y) \frac{M(g\chi_{Q})(y)^{\frac1n}\,M(g\chi_{Q})(x)^{\frac1{n'}}}{\lvert x-y\rvert^{n-(1-\delta)}}\,\dif y\,\dif \mu(x)\\
&= I+II.
\end{split}
\]
We work first with $I$. By Tonelli's theorem,
\begin{align*}
I & \leq 
\int_{Q} M(g\chi_{Q})(x)^{\frac1n}\, 
\int_{Q} \frac{M(g\chi_{Q})(y)^{\frac1{n'}}}{\lvert x-y\rvert^{n-(1-\delta)}}\,\dif y\,\dif \mu(x)\\
&=\int_{Q} M(g\chi_{Q})(x)^{\frac1n}\, I_{1-\delta} (\chi_{Q} M(g\chi_{Q})^{\frac1{n'}} )(x) 
\,\dif \mu(x).
\end{align*}
By the Coifman--Rochberg lemma \cite[pp. 158--159]{GCRdF} we know that $M(g\chi_{Q})(y)^{\frac1{n'}}\in A_1$ with a constant depending on the dimension. In particular, we get
\[M^c(\chi_QM(g\chi_{Q})^{\frac1{n'}})\le C(n)M(g\chi_{Q})^{\frac1{n'}}\]
almost everywhere in $Q$.  We use this estimate together with 
inequality \eqref{lem-eq1} in Lemma \ref{New-pointwise-B} to get
\begin{align*}
I&\leq \frac{C(n)}{1-\delta}\left( \int_{Q} M(g\chi_{Q})(x)^{\frac1{n'}} \,\dif x\right)^{\frac{1-\delta}{n} }
\int_{Q} M(g\chi_{Q})(x)^{\frac1n}\, 
M(g\chi_{Q})(x)^{\frac{n-(1-\delta)}{nn'} }
\,\dif \mu(x)\\
&= \frac{C(n)}{1-\delta}\ell(Q)^{1-\delta} \,\left( \vint_{Q} M(g\chi_{Q})(x)^{\frac1{n'}} \,\dif x\right)^{\frac{1-\delta}{n} }
\int_{Q} M(g\chi_{Q})(x)^{1-\frac{1-\delta}{nn'}}\,\dif \mu(x).
\end{align*}
To estimate the last two integrals we use Kolmogorov's inequality \eqref{kolmogorov}. By using also the weak type $(1,1)$
estimate for the maximal operator, we estimate the first integral
\[
\vint_{Q} M(g\chi_{Q})(x)^{\frac1{n'}} \,\dif x \le C(n)\| M(g\chi_{Q}) \|^\frac{1}{n'}_{L^{1,\infty}\big(Q, \frac{\dif x}{\lvert Q\rvert}\big)} 
\leq C(n)\left( \vint_{Q} g\chi_{Q}(x) \,\dif x\right)^{\frac1{n'}}.
\]
We estimate the second integral using Kolmogorov's inequality \eqref{kolmogorov} and the Fefferman--Stein  inequality \eqref{weightedweaktype}, thus getting
\begin{align*}
\frac{1}{\mu(Q)} \int_{Q} \, 
M(g\chi_{Q})(x)^{1-\frac{1-\delta}{nn'}}\, 
\,\dif \mu(x)
&\leq 
\frac{C(n)}{1-\delta}\, \| M(g\chi_{Q})\|^{1-\frac{1-\delta}{nn'}}_{L^{1,\infty}\big(Q, \frac{\dif \mu}{\mu(Q)}\big)}\\
&\leq
\frac{C(n)}{1-\delta}\,  \left( \sup_{t>0}\frac{t}{\mu(Q)}  \mu\{x\in \R^n: M(g\chi_{Q})(x)>t\} \right)^{1-\frac{1-\delta}{nn'}}
\\&\leq
\frac{C(n)}{1-\delta}\,  \left( \frac{1}{\mu(Q)}  \int_{Q} g(x)\,M\mu(x)\,\dif x \right)^{1-\frac{1-\delta}{nn'}}.
\end{align*}
We have shown that
\begin{equation*}
I \le \frac{C(n)}{(1-\delta)^2}\ell(Q)^{1-\delta} \, \left( \vint_{Q} g\chi_{Q}(x) \,\dif x\right)^{\frac{1-\delta}{nn'}}
\left( \frac{1}{\mu(Q)}  \int_{Q} g(x)\,M\mu(x)\,\dif x \right)^{1-\frac{1-\delta}{nn'}}\mu(Q)\,.
\end{equation*}
for Borel measures $\mu$ in $\R^n$.
If we choose the measure $\mu$ to be defined through an $A_1$ weight $w$, we get
\begin{align*}
I&\leq 
\frac{C(n)}{(1-\delta)^2}\ell(Q)^{1-\delta} \, \left( \vint_{Q} g\chi_{Q}(x) \,\dif x\right)^{\frac{1-\delta}{nn'}}
\left( \frac{1}{w(Q)}  \int_{Q} g(x)\,Mw(x)\,\dif x \right)^{1-\frac{1-\delta}{nn'}}\, w(Q)\\
&\leq 
\frac{C(n)}{(1-\delta)^2}\ell(Q)^{1-\delta} \, \left( \frac{[w]_{A_1}}{w(Q)}\int_{Q} g (x) \,w(x)\,\dif x\right)^{\frac{1-\delta}{nn'}}
\left( \frac{[w]_{A_1}}{w(Q)}  \int_{Q} g(x)\,w(x)\,\dif x \right)^{1-\frac{1-\delta}{nn'}}\, w(Q)\\
&\leq 
\frac{C(n)\,[w]_{A_1}}{(1-\delta)^2} \,  \ell(Q)^{1-\delta} \,\int_{Q} g (x) \,w(x)\,\dif x.
\end{align*}
Since the term $II$ can be treated in the same way, we have finished the proof
of the theorem.

%
%
%


\subsection{Proof of Theorem \ref{weightedBBMp=1-weak}}

We adapt the non-weighted proof from \cite{BVSY}.

We have to prove that 
\begin{equation*}
\bigg{\|} \frac{u(x)-u(y)}{\lvert x-y\rvert^{n+\delta}} \bigg{\|}_{L^{1,\infty}\big(Q \times Q,  \dif w(x) \times \dif y\big)}
\leq C(n)\, %
 \frac{[w]_{A_1}^{ 2+\frac{1-\delta}{n}} }{ \delta(1-\delta)}   
\ell(Q)^{1-\delta}   \int_Q |\nabla u(x)|w(x)\,\dif x.
\end{equation*}
By  the  $(1,1)$-Poincar\'e inequality \eqref{1-pPI}, there exists a dimensional constant
$C(n)>0$ such that
\[\vint_{R}\lvert u(x)-u_{R}\rvert\,\dif x \leq  c_n\, \ell(R)\vint_R \lvert \nabla u(x)\rvert\,\dif x
\]
for every cube $R\subset Q$.
We use Lemma \ref{l.loc_impro} 
with $g=|\nabla u|$,\, $\kappa=C(n)$
and $0<\beta=\delta<1=\alpha$.  Hence, for any pair of Lebesgue points $x,y\in Q$ of $u$, the inequality  
\[
|u(x)-u(y)|\leq  \,\frac{C(n)}{\delta}\,|x-y|^{\delta} (M_{1-\delta,Q}(g\chi_Q)(x)+M_{1-\delta,Q}(g\chi_Q)(y))
\]
holds.

Thus, we can estimate
\[
\begin{split}
&\bigg{\|} \frac{u(x)-u(y)}{\lvert x-y\rvert^{n+\delta}} \bigg{\|}_{L^{1,\infty}\big(Q \times Q,  \dif w(x) \times \dif y\big)}\\
&\leq \frac{C(n)}{\delta}\,
\bigg{\|} \frac{M_{1-\delta,Q}g(x)}{\lvert x-y\rvert^{n}} \bigg{\|}_{L^{1,\infty}\big(Q \times Q,  \dif w(x) \times \dif y\big)}+
\frac{C(n)}{\delta}\, \bigg{\|} \frac{M_{1-\delta,Q}g(y)}{\lvert x-y\rvert^{n}} \bigg{\|}_{L^{1,\infty}\big(Q \times Q,  \dif w(x) \times \dif y\big)}\\&=\frac{C(n)}{\delta}\,(I+II).
\end{split}
\]

In order to estimate $I$, we write
\[
E_t=  \left\{ (x,y):   \frac{M_{1-\delta,Q}g(x)}{\lvert x-y\rvert^{n}}>t      \right\}
\]
for every $t>0$.
Then, by definition of the weak quasinorm and Tonelli's theorem,
\[
I= \sup_{t>0} t \int_{Q} \int_{Q}  \car {E_t}( x,y) w(x)\,\dif x\,\dif y =\sup_{t>0} t \int_{Q} \int_{Q}  \car {E_t}( x,y) \dif y\, w(x)\,\dif x\,.
\]
For a fixed $x$, we easily see that $\chi_{E_t}(x,\cdot)$ is 
the characteristic function of 
the ball centered at $x$ with radius  
\[
\left( \frac{M_{1-\delta,Q}g(x)}{t} \right)^{ \frac{1}{{n}}},
\]
and hence       
\begin{equation}\label{endI}
I\leq C(n)\sup_{t>0} \,t \int_{Q}   \frac{M_{1-\delta,Q}g(x)}{t}    w(x)\,\dif x
=  C(n)\int_{Q}   M_{1-\delta,Q}g(x)  w(x)\,\dif x.
\end{equation}
Using Lemma \ref{l.riesz_A1}, we finally get
$$
I \leq C(n)  \frac{[w]_{A_1}^{ 1+\frac{1-\delta}{n}} }{1-\delta}   \ell(Q)^{1-\delta} \int_Q g(x) \, w(x)\,\dif x,
$$
and this yields the estimate for $I$.

It remains to estimate the second term
\[
II= \sup_{t>0} \,t \int_{Q} \int_{Q}  \car {F_t}( x,y) w(x)\,\dif x\,\dif y,
\]
where %
\[
 F_t =  \left\{ (x,y):   \frac{M_{1-\delta,Q}g(y)}{\lvert x-y\rvert^{n}}>t     \right\}\,.
\]
For a fixed $y$, we see that $\chi_{F_t}(\cdot,y)$ is the
characteristic function of the  ball centered at $y$ and with radius  
\[
 r_{y,t}= \left( \frac{M_{1-\delta,Q}g(y)}{t} \right)^{\frac1n}.
\]
Hence,
\begin{align*}
II&\leq \sup_{t>0} t \int_{Q}  \, w(B(y, r_{y,t}))\,\dif y\\
&\leq \sup_{t>0} t \int_{Q}  \, \frac{ w(Q(y, r_{y,t})) }
{\lvert Q(y, r_{y,t})\rvert}\,  \lvert Q(y, r_{y,t})\rvert\,\dif y\\
&\leq  c_n  \int_{Q}  M^{c}w(y)\, M_{1-\delta,Q}g(y) \,\dif y 
\leq  c_n  [w]_{A_1} \int_{Q}  M_{1-\delta,Q}g(y) \,w(y)\dif y\,.
\end{align*}
Using Lemma \ref{l.riesz_A1}, 
$$
II \leq c_n  \frac{[w]_{A_1}^{ 2+\frac{1-\delta}{n}} }{1-\delta}   \ell(Q)^{1-\delta} \int_Q g(x) \, w(x)\,\dif x,
$$
and this yields the estimate for $II$. This concludes the proof of the theorem. 

%

\section{Appendix }\label{Appendix}

In this appendix  we prove
inequality \eqref{e.A}  which appears already 
in \cite[Theorem 1]{BBM1}.
Our proof is based on very elementary calculations. 
The case $p=1$ of this approach appears  in \cite{CMPR}.

\begin{lemma}\label{l.fractional_functional}
 Let $Q$ be a cube in $\R^n$
and let $u\in C^1(Q)$.  Let $0<\delta<1$  and  $1\leq p<\infty$.
There is a dimensional constant $C(n)>0$ such that 
\begin{equation}\label{fractional_lhs}
\left(\vint_Q \int_Q\frac{|u(x)-u(y)|^p}{|x-y|^{n+\delta p}}\,\dif y\,\dif x\right)^{\frac{1}{p}} \leq 
\frac{C(n)}{ \alpha(\delta,p) } \ell(Q)^{1-\delta}\left(\vint_Q |\nabla u(x)|^p\,\dif x\right)^{\frac{1}{p}}.
\end{equation}
Furthermore, %
$$
\alpha(\delta,p)=\left\lbrace\begin{array}{rcl}   (1-(1-\delta)p)^{\frac1p} \, ((1-\delta)p)^{\frac1p} &if&  (1-\delta)p<1, \\ 
((1-\delta)p-1)^{\frac1p}  &if& (1-\delta)p>1,
\\  1        &if& (1-\delta)p=1  \end{array}\right.
$$
and hence $\alpha(\delta,p) \approx (1-\delta)^{\frac1p}$ as $\delta\to1$. 

\end{lemma}

\begin{proof}
Let us first prove $(1)$.
By the Fundamental Theorem of Calculus one can write, for every $x,y\in Q$, 
\[
u(y)-u(x)=\int_0^1\nabla u(x+t(y-x))\cdot(y-x)\,\dif t,
\]
where $\cdot$ represents the usual scalar product in $\mathbb{R}^n$. 
Thus, using this equality, H\"older's inequality and Tonelli's theorem,

\begin{equation}
\begin{aligned}
&\left(\vint_Q\int_Q\frac{|u(x)-u(y)|^p}{|x-y|^{n+\delta p}}\,\dif y\,\dif x\right)^{\frac{1}{p}}\\
&\qquad \leq \left(\vint_Q\int_Q\int_0^1\frac{|\nabla u(x+t(y-x))|^p}{|x-y|^{n+\delta p-p}}\dif t\,\dif y\,\dif x\right)^{\frac{1}{p}}\\
&\qquad  =\left(\vint_Q\int_0^1\int_{Q\cap B(x,\sqrt{n}\ell(Q))}\frac{|\nabla u(x+t(y-x))|^p}{|x-y|^{n-(1-\delta)p}}\,\dif y\,\dif t\,\dif x\right)^{\frac{1}{p}},
\end{aligned}
\end{equation}
since $Q\subset B(x,\sqrt{n}\ell(Q))$ for any $x\in Q$. By change of variables $z=x+t(y-x)=(1-t)x+ty$, one has
\begin{enumerate}
\item By convexity, $x,y\in Q$ implies $z\in Q$, so $\chi_Q(y)=\chi_Q(z)$.
\item $|x-y|=|z-x|/t$.
\end{enumerate}
Thus, we continue with 
\begin{equation}
\begin{aligned}
& \qquad=\left(\vint_Q\int_0^1\int_{ ((1-t)x+tQ) \cap B(x,\sqrt{n}t\ell(Q))} \frac{|\nabla u(z)|^p}{|z-x|^{n-(1-\delta)p}}\frac{t^{n-(1-\delta)p}}{t^n}\,\dif z\,\dif t\,\dif x\right)^{\frac{1}{p}}\\
&\qquad \leq   \left(\vint_Q\int_0^1\int_{Q\cap B(x,\sqrt{n}t\ell(Q))} \frac{|\nabla u(z)|^p}{|z-x|^{n-(1-\delta)p}}t^{-(1-\delta)p}\,\dif z\,\dif t\,\dif x\right)^{\frac{1}{p}}\\
&\qquad =\left(\vint_Q\int_Q\int_{\frac{|z-x|}{\sqrt{n}\ell(Q)}}^1 \frac{\dif t}{t^{(1-\delta)p}} \frac{|\nabla u(z)|^p}{|z-x|^{n-(1-\delta)p}}\,\dif z\,\dif x\right)^{\frac{1}{p}},
\end{aligned}
\end{equation}
where we used Tonelli's theorem again in the last equality.

There are three possibilities depending on the value of $(1-\delta)p$.

{\bf Case 1.}  If $(1-\delta)p<1$, then we can extend the integral to zero, and we get the following upper bound 
\begin{equation}\label{<1}
\begin{aligned}
&  \frac{1}{(1-(1-\delta)p)^{\frac1p}  } \,  \left(\vint_Q\int_Q   \frac{|\nabla u(z)|^p}{|z-x|^{n-(1-\delta)p}}\,\dif z\,\dif x\right)^{\frac{1}{p}}\\
&\qquad =  \frac{1}{(1-(1-\delta)p)^{\frac1p}  } \, \left(\vint_Q|\nabla u(z)|^p \int_Q  \frac{\dif x}{|z-x|^{n-(1-\delta)p}}\,\dif z\right)^{\frac{1}{p}}\\
&\qquad \leq  \frac{1}{(1-(1-\delta)p)^{\frac1p}  } \,\left(\vint_Q|\nabla u(z)|^p \frac{C(n)\ell(Q)^{(1-\delta)p}}{v_n^{(1-\delta)p/n}(1-\delta)p}\dif z\right)^{\frac{1}{p}}\\
&\qquad \le \frac{C(n)\ell(Q)^{1-\delta}}{v_n^{\frac{1-\delta}{n}}((1-\delta)p)^{\frac{1}{p}} (1-(1-\delta)p)^{\frac1p} }\left(\vint_Q|\nabla u(x)|^p\,\dif x\right)^{\frac{1}{p}},
\end{aligned}
\end{equation}
where we used Tonelli's theorem and the  fact that, for a Lebesgue measurable set $E$ and $0<\alpha<n$,
 \begin{equation}\label{lemita}
\int_E \frac{\dif z}{|x-z|^{n-\alpha}} \le C(n) v_n^{-\frac{\alpha}{n}}\alpha^{-1}|E|^{\frac{\alpha}{n}},\qquad \text{for all }x\in\R^n,
\end{equation}
where $v_n$ is the volume of the unit ball of $\mathbb{R}^n$.
 See Lemma \ref{New-pointwise-B}. 

{\bf Case 2.} If $(1-\delta)p>1$, then we can extend the upper bound of the integral in $t$ up to infinity and then compute it. This way we obtain the following bound
\begin{equation}
\begin{aligned}
& \left(\vint_Q\int_Q  \frac{|z-x|^{1-(1-\delta)p}}{((1-\delta)p-1)(\sqrt n\ell(Q))^{1-(1-\delta)p}} \frac{|\nabla u(z)|^p}{|z-x|^{n-(1-\delta)p}}\,\dif z\,\dif x\right)^{\frac{1}{p}}\\
&\qquad \leq  \frac{\sqrt n}{((1-\delta)p-1)^{\frac1p}} \ell(Q)^{-\frac{1}{p}+1-\delta} \left(\vint_Q|\nabla u(z)|^p\int_Q   \frac{\dif x}{|z-x|^{n-1}} \,\dif z\right)^{\frac{1}{p}}\\
&\qquad \leq    \frac{\sqrt n}{((1-\delta)p-1)^{\frac1p}}\,\ell(Q)^{ -\frac{1}{p}+1-\delta } \left(\vint_Q    |\nabla u(z)|^p  C(n) \ell(Q)\dif z\right)^{\frac{1}{p}}\\
&\qquad \leq \frac{C(n)}{((1-\delta)p-1)^{\frac1p}}\,\ell(Q)^{1-\delta }\left(\vint_Q    |\nabla u(z)|^p \dif z\right)^{\frac{1}{p}},
\end{aligned}
\end{equation}
where again we used Tonelli's theorem and inequality \eqref{lemita}.

{\bf Case 3.} If $(1-\delta)p=1$, then we compute the integral in $t$ and use the elementary inequality 
\[\log s\le \frac{s^{q}-1}{q}\leq \frac{s^q}{q}\] which holds whenever $s>1$ and $q>0$. Applying this, say,  with $q=\frac{1}{2}$ we get the upper bound
\begin{equation}
\begin{aligned}
& \left(\vint_Q\int_Q  \log\left(\frac{\sqrt{n}\ell(Q)}{|z-x|}\right)\frac{|\nabla u(z)|^p}{|z-x|^{n-1}}\dif z\dif x\right)^{\frac{1}{p}}\\
&\qquad \le  C(n) \ell(Q)^{\frac{q}{p}} \left(\vint_Q|\nabla u(z)|^p\int_Q   \frac{\dif x}{|z-x|^{n-(1-q)}}\dif z\right)^{\frac{1}{p}}\\
&\qquad \le  C(n)\, \ell(Q)^{ \frac{q}{p}+\frac{(1-q)}{p} } \left(\vint_Q    |\nabla u(z)|^p \dif z\right)^{\frac{1}{p}}\\
&\qquad =C(n)\,\ell(Q)^{1-\delta}\left(\vint_Q    |\nabla u(z)|^p \dif z\right)^{\frac{1}{p}},
\end{aligned}
\end{equation}
where Tonelli's theorem and inequality \eqref{lemita} have been used one more time.
\end{proof}%

\begin{remark}
Let $Q$ be a cube in $\R^n$ and let $u\in L^1(Q)$. Let $0<\beta\le \delta<1$ and $1\le p<\infty$. Then it is straightforward to show that
\[
\ell(Q)^\beta\left(\vint_Q \int_Q\frac{|u(x)-u(y)|^p}{|x-y|^{n+\beta p}}\,\dif y\,\dif x\right)^{\frac{1}{p}} \le 
C(n)\ell(Q)^\delta \left(\vint_Q \int_Q\frac{|u(x)-u(y)|^p}{|x-y|^{n+\delta p}}\,\dif y\,\dif x\right)^{\frac{1}{p}}\,.
\]
\end{remark}

\bibliographystyle{abbrv}

\end{document}